\documentclass[12pt]{amsart}
\usepackage {amsxtra,amssymb,amsmath,amsfonts,amscd}

\newtheorem{theorem}{Theorem}[section]
\newtheorem{proposition}[theorem]{Proposition}

\newtheorem{lemma}[theorem]{Lemma}

\newtheorem{definition}[theorem]{Definition}

\newtheorem{corollary}[theorem]{Corollary}

\theoremstyle{plain}
\numberwithin{equation}{theorem}

\theoremstyle{remark}

\newcommand{\Q}{{\mathbb Q}}

\newcommand{\R}{{\mathbb R}}
\newcommand{\Z}{{\mathbb Z}}

\renewcommand{\L}{{\mathcal L}}

\newcommand{\Pl}{{\mathbb P}}

\DeclareMathOperator{\Pic}{Pic}

\DeclareMathOperator{\Div}{div}

\DeclareMathOperator{\var}{Var}
\DeclareMathOperator{\Span}{Span}

\newcommand{\lra}{\longrightarrow}
\newcommand{\cO}{\mathcal{O}}

\title[Four-folds with involutions]{Geometry of Four-folds with three non-commuting involutions}
\author{Jorge Pineiro}

\address{Department of Mathematics and Computer Science.
Bronx Community College of CUNY.
2155 University Ave.
Bronx, NY 10453}
\email{jorge.pineiro@bcc.cuny.edu}
\subjclass[2010]{Primary:  37P55  ; Secondary: 14J35, 14J28, 32J15 }

\begin{document}
\begin{abstract}
In this paper we adapt some techniques developed for K3 surfaces, to study the geometry of a family of projective varieties in $\Pl_K^2 \times \Pl_K^2 \times \Pl_K^2$ defined as the intersection of a form of degree $(2,2,2)$ and a form of degree $(1,1,1)$.  Members of the family will be equipped with dominant rational self-maps and we will study the actions of those maps on divisors and compute the first dynamical degrees of the composition of any pair.
\end{abstract}

\maketitle

 \section{Introduction}
As a generalization of the work of Silverman and others \cite{Sil-K3}, \cite{Comput_on_K3} on families of K3 surfaces with infinite groups of automorphisms, we study dynamics on a family of varieties $X^{A,B}$ in $\Pl_K^2 \times \Pl_K^2 \times \Pl_K^2$ defined as the intersection of a form of degree $(2,2,2)$ and a form of degree $(1,1,1)$.  Individual members of the family $X^{A,B}$ come equipped with $(2:1)$-projections $p_1,p_2,p_3 : X^{A,B} \lra \Pl^2 \times \Pl^2$ that generate involutions $\sigma_1,\sigma_2,\sigma_3$ on $X^{A,B}$.  In this situation however the maps $\sigma_i$ for $i=1,2,3$ are not morphisms of the whole $X^{A,B}$, but only rational dominant maps. Still it is possible to induce maps $\sigma_i^* : \Pic(X) \lra \Pic(X)$ and $\tilde{\sigma}_i^* : NS(X)_{\Q} \lra NS(X)_{\Q}$, on divisors modulo linear and numerical equivalence.\\  The following degree associated to the dynamics was initially studied by Arnold in \cite{arnold}, and particularly for dominant rational maps by Silverman in \cite{Sil-Dynamical-degree}.
\begin{definition}
Let $X$ be an algebraic variety and $\varphi : X \dashrightarrow X$ a dominant rational map.  The first dynamical degree of $\varphi$ is $$\delta_{\varphi}=\limsup_{n \rightarrow \infty} \rho(\widetilde{\varphi^{n*}})^{1/n},$$
where $\rho(\widetilde{\varphi^{n*}})$ represents the spectral radius or maximal eigenvalue of the map $\widetilde{\varphi^{n*}} : NS(X)_{\Q} \lra NS(X)_{\Q}$.
\end{definition}
\noindent It is also possible to extend the notion of polarization, with respect to one rational map or, more general, in the sense of Kawaguchi \cite{kawaguchi}, associated to several rational maps:
\begin{definition}
Let $X$ be an algebraic variety and $\varphi_i : X \dashrightarrow X$ for $i=1,\dots,k$ dominant rational maps.  We say that the system $(X,\{\varphi_1,\dots,\varphi_k\},\L,d)$ is a polarized dynamical system of $k$ maps if there exist an ample line bundle $\L$ on $X$ such that $\bigotimes_{i=1}^k \varphi^*_i \L \cong \L^{d}$ for some $d > k$.
\end{definition}
The action of the maps $\sigma_1^*$, $\sigma_2^*$ and $\sigma_3^*$ on $\Pic(X)$ will provide a polarization for the system of three maps $\{\sigma_1,\sigma_2,\sigma_3\}$.  Also, under the condition that the Picard number is the least possible value $p(X)=3$, the first dynamical degree of any of the maps $\sigma_{ij}=\sigma_i \circ \sigma_j$ will be computed.  The computations will produce the same dynamical degree as the dynamical degree of the maps on K3 surfaces (Section 12 of \cite{Sil-Dynamical-degree}).

 \section{Four dimensional Varieties with three involutions} Let $\textbf{L}^{A} \subset \Pl^2 \times \Pl^2 \times \Pl^2$ be a family of varieties defined over a field $K$ by a single equation linear on each variable,
$$ \textbf{L}^A= \{ P \in \Pl^2 \times \Pl^2 \times \Pl^2 : L(x,y,z)=\sum^2_{i,j,k=0}a_{i,j,k}x_i y_jz_k=0\},$$
where $A=(a_{ijk})_{0 \leq i,j,k \leq 2}$.  A member of the family \textbf{L} comes equipped with projections
$$p_3=p_{xy} : \textbf{L} \lra \Pl^2 \times \Pl^2, $$
$$ p_2=p_{xz} : \textbf{L} \lra \Pl^2 \times \Pl^2, $$
$$ p_1=p_{yz} : \textbf{L} \lra \Pl^2 \times \Pl^2.$$
and the $\Pic(\textbf{L}) \cong \Z^3$ from the embedding $\textbf{L} \hookrightarrow \Pl^2 \times \Pl^2 \times \Pl^2$.  Using the adjunction formula we can get its canonical line bundle $$\omega_\textbf{L} \cong \cO_{\Pl^2 \times \Pl^2 \times \Pl^2}(-3,-3,-3) \otimes \cO_{\Pl^2 \times \Pl^2 \times \Pl^2}(\textbf{L})=\cO_{\Pl^2 \times \Pl^2 \times \Pl^2}(-2,-2,-2).$$
By choosing a section $Q=Q^{A} $ of $ \cO_{\textbf{L}}(2,2,2)$ and consider the variety $X=\var(Q)$ we get a variety with trivial canonical divisor $K_X \sim 0$.  Besides, by the weak lefschetz theorem, we have an injective map $\Z^3 \cong \Pic(L) \hookrightarrow \Pic(X)$ and we will get three distinct classes even in $NS(X)$ and therefore a Picard number $p(X) \geq 3$.  \\By varying the coefficients $A,B$ one obtains a family $X^{A,B}$ defined in $\Pl^2_K \times \Pl^2_K \times \Pl^2_K$ by equations
\begin{align*}
L(x,y,z)& =\sum^2_{i,j,k=0}a_{i,j,k}x_i y_jz_k=0,\\
 Q(x,y,z) &=\sum^2_{i,j,k,l,m,n=0}b_{i,j,k,l,m,n}x_ix_l y_j
y_m z_k z_n=0,
\end{align*}
where $A=(a_{ijk})$, $B=(b_{i,j,k,l,m,n})$ and all indices are moving in the set $\{0,1,2\}$.  The projections $p_1,p_2,p_3$ restricted to $X$ represent generically $(2:1)$ coverings of $\Pl^2 \times \Pl^2$.  Indeed when we fix two of the variables we get the intersection on $\Pl^2$ of a quadric and a line, which is general, will give two points $P_i,P_i' \in X$ for $i=1,2,3$ and will determine involutions $\sigma_1, \sigma_2,\sigma_3 : X \dashrightarrow X$.  The involutions $\sigma_i$ for $i=1,2,3$, will not be in general morphisms but just rational maps defined on certain open sets $U_i \subset X$.  We are interesting in studying the dynamics of the maps $\sigma_i$, but first we should devote some time to get familiar with the geometry of $X=X^{A,B}$.  We collect the coefficients of our variables using the following notation for $i,j,k$ in the set $\{0,1,2\}$
\begin{align*}
L_k^{x,y}(x,y) & =\sum^2_{i,j=0}a_{i,j,k}x_i y_j,  & Q^{x,y}_{k,n}(x,y)=\sum^2_{i,j,l,m=0}b_{i,j,k,l,m,n}x_ix_l y_j y_m, \\
L_j^{x,z}(x,z) & =\sum^2_{i,k=0}a_{i,j,k}x_i z_k,  & Q^{y,z}_{i,l}(y,z)=\sum^2_{j,k,m,n=0}b_{i,j,k,l,m,n} y_jy_m z_k z_n,\\
L_i^{y,z}(y,z) & =\sum^2_{j,k=0}a_{i,j,k} y_j z_k, & Q^{x,z}_{j,m}(x,z)=\sum^2_{i,k,l,n=0}b_{i,j,k,l,m,n}x_ix_l z_k z_n.
\end{align*}
Suppose, with the above notation in mind, that we want to study the action of $\sigma_3$ computing the solutions $(z_0,z_1,1)$ of the system
\begin{align*}
0= & L_0^{x,y}z_0+L_1^{x,y}z_1+L_2^{x,y}, \\
0= & Q^{x,y}_{0,0}z_0^2+Q^{x,y}_{1,1}z_1^2+Q^{x,y}_{2,2}+Q^{x,y}_{0,1}z_0z_1+Q^{x,y}_{0,2}z_0+Q^{x,y}_{1,2}z_1,
\end{align*}
assuming that $L_1^{x,y} \neq 0$ and replacing $z_1=\frac{-L_2^{x,y}-L_0^{x,y}z_0}{L_1^{x,y}}$ in the second equation gives
$G_0^{x,y} + H^{x,y}_{0,2} z_0 + G_2^{x,y}z_0^2=0$
where,
$$G_0^{x,y}=(L_1^{x,y})^2Q^{x,y}_{2,2}-L_1^{x,y}L_2^{x,y}Q^{x,y}_{1,2}+(L_2^{x,y})^2Q^{x,y}_{1,1},$$ $$G_2^{x,y}=(L_1^{x,y})^2Q^{x,y}_{0,0}-L_1^{x,y}L_0^{x,y}Q^{x,y}_{0,1}+(L_0^{x,y})^2Q^{x,y}_{1,1},$$ $$H^{x,y}_{0,2}=2L^{x,y}_0L^{x,y}_2Q_{1,1}^{x,y}-L^{x,y}_0L_1^{x,y}Q^{x,y}_{1,2}-L^{x,y}_2L_1^{x,y}Q_{0,1}^{x,y}+(L_1^{x,y})^2Q^{x,y}_{0,2},$$

and the map $\sigma_3$ that sends $(z_0,z_1,1) \mapsto (z'_0,z'_1,1)$ will be defined unless all the three coefficients $G_0^{x,y}, H^{x,y}_{0,2},G_2^{x,y}$ vanish.  So, we are forced, by a codimension checking, to work with rational maps $\sigma_i : X \dashrightarrow X$ and our first task will be, to locate where are these maps well defined morphisms. \\ Motivated by the above discussion we define for any permutation $(i,j,k)$ of $(0,1,2)$ the $(4,4)$-bi-homogeneous forms
$$G_k^{x,y}=(L_i^{x,y})^2Q^{x,y}_{j,j}-L_i^{x,y}L_j^{x,y}Q^{x,y}_{i,j}+(L_j^{x,y})^2Q^{x,y}_{i,i},$$
$$G_k^{y,z}=(L_i^{y,z})^2Q^{y,z}_{j,j}-L_i^{y,z}L_j^{y,z}Q^{y,z}_{i,j}+(L_j^{y,z})^2Q^{y,z}_{i,i},$$
$$G_k^{x,z}=(L_i^{x,z})^2Q^{x,z}_{j,j}-L_i^{x,z}L_j^{x,z}Q^{x,z}_{i,j}+(L_j^{x,z})^2Q^{x,z}_{i,i},$$
$$H^{x,y}_{i,j}=2L^{x,y}_iL^{x,y}_jQ_{kk}^{x,y}-L^{x,y}_iL_k^{x,y}Q^{x,y}_{jk}-L^{x,y}_jL_k^{x,y}Q_{ik}^{x,y}+(L_k^{x,y})^2Q^{x,y}_{ij},$$
$$H^{x,z}_{i,j}=2L^{x,z}_iL^{x,z}_jQ_{kk}^{x,z}-L^{x,z}_iL_k^{x,z}Q^{x,z}_{jk}-L^{x,z}_jL_k^{x,z}Q_{ik}^{x,z}+(L_k^{x,z})^2Q^{x,z}_{ij},$$
$$H^{y,z}_{i,j}=2L^{y,z}_iL^{y,z}_jQ_{kk}^{y,z}-L^{y,z}_iL_k^{y,z}Q^{y,z}_{jk}-L^{y,z}_jL_k^{y,z}Q_{ik}^{y,z}+(L_k^{y,z})^2Q^{y,z}_{ij},$$

For any $a,b,c \in \Pl^2_K$, the fibres of the projections $p_1, p_2$ and $p_3$ will be defined as $ X_{a,b}^z = p_3^{-1}(a,b)=L_{a,b}^z \cap  Q_{a,b}^z$, $ X_{b,c}^x = p_1^{-1}(b,c)=L_{b,c}^x \cap  Q_{b,c}^x$ and $X_{a,c}^y = p_2^{-1}(a,c)=L_{a,c}^y \cap  Q_{a,c}^y$; where
$$L_{a,b}^z=\{(a,b,z) \in \Pl^2 \times \Pl^2 \times \Pl^2 : L(a,b,z)=0   \},$$
 $$ Q_{a,b}^z=\{(a,b,z) \in \Pl^2 \times \Pl^2 \times \Pl^2 : Q(a,b,z)=0$$
$$L_{b,c}^x=\{(x,b,c) \in \Pl^2 \times \Pl^2 \times \Pl^2 : L(x,b,c)=0   \},$$
 $$Q_{b,c}^x=\{(x,b,c) \in \Pl^2 \times \Pl^2 \times \Pl^2 : Q(x,b,c)=0$$
$$L_{a,c}^y=\{(a,y,c) \in \Pl^2 \times \Pl^2 \times \Pl^2 : L(a,y,c)=0   \}, $$
 $$Q_{a,c}^y=\{(a,y,c) \in \Pl^2 \times \Pl^2 \times \Pl^2 : Q(a,y,c)=0.$$

\begin{definition}
We say that a fibre $ X_{a,b}^z, X_{b,c}^x$ or $X_{a,c}^y$ is degenerate if it has positive dimension.
\end{definition}
If the fibres $ X_{a,b}^z, X_{b,c}^x$ or $X_{a,c}^y$ are non-degenerate at $(a,b,c)$, they will consist of two points and the maps $\sigma_1, \sigma_2$ and $\sigma_3$ will be well defined morphisms at $(a,b,c) \in X$.  Following the outline of \cite{Comput_on_K3} we have the following result characterizing the degenerate fibres.
\begin{proposition} Let $[a,b,c] \in X$. \label{degenerate}
\begin{enumerate}
\item[(1)] $X_{a,b}^z$ is degenerate if and only if $$G_0^{x,y}(a,b)=G_1^{x,y}(a,b)=G_2^{x,y}(a,b)=H_{0,1}^{x,y}(a,b)=H_{0,2}^{x,y}(a,b)=H_{1,2}^{x,y}(a,b)=0.$$
\item[(2)] $X_{a,c}^y$ is degenerate if and only if $$G_0^{x,z}(a,c)=G_1^{x,z}(a,c)=G_2^{x,z}(a,c)=H_{0,1}^{x,z}(a,c)=H_{0,2}^{x,z}(a,c)=H_{1,2}^{x,z}(a,c)=0.$$
\item[(3)] $X_{b,c}^x$ is degenerate if and only if $$G_0^{y,z}(b,c)=G_1^{y,z}(b,c)=G_2^{y,z}(b,c)=H_{0,1}^{y,z}(b,c)=H_{0,2}^{y,z}(b,c)=H_{1,2}^{y,z}(b,c)=0.$$
\end{enumerate}
\end{proposition}
\begin{proof} The proof is identical to the proof of proposition 1.4 in \cite{Comput_on_K3}.  We do the proof of (1).  When we substitute $z_0=(L-L_1^{x,y}z_1-L_2^{x,y}z_2)/L_0^{x,y}, z_1=(L-L_0^{x,y}z_0-L_2^{x,y}z_2)/L_1^{x,y}$ and $ z_2=(L-L_1^{x,y}z_1-L_0^{x,y}z_0)/L_2^{x,y}$ into $Q$ respectively we get formulas:
$$(L_0^{x,y})^2Q(x,y,z) \equiv G_2^{x,y} z_1^2 + H_{1,2}^{x,y} z_1 z_2 + G_1^{x,y} z_2^2 \quad (mod L(x,y,z)),$$
$$(L_1^{x,y})^2Q(x,y,z) \equiv G_2^{x,y} z_0^2 + H_{0,2}^{x,y} z_0 z_2 + G_0^{x,y} z_2^2 \quad (mod L(x,y,z)),$$
$$(L_2^{x,y})^2Q(x,y,z) \equiv G_1^{x,y} z_0^2 + H_{0,1}^{x,y} z_0 z_1 + G_0^{x,y} z_1^2 \quad (mod L(x,y,z)).$$
Now, the proof is divided into two parts, depending on whether or not for the point $[a,b,c] \in \Pl^2 \times \Pl^2 \times \Pl^2$ we have $L(a,b,z) \equiv 0$.\\  If $L(a,b,z) \equiv 0$, then $ X_{a,b}^z = Q_{a,b}^z $ and the fibre is degenerate.  In this case $L_0^{a,b}=L_1^{a,b}=L_2^{a,b}=0$ will force $H^{x,y}_{i,j}(a,b)=G_k^{x,y}(a,b)=0$ and the proof is finished.\\
If $L(a,b,z) \neq 0$, one of the $L_i^{x,y}(a,b) \neq 0$ and the fact that $G_0^{x,y}(a,b)=G_1^{x,y}(a,b)=G_2^{x,y}(a,b)=H_{0,1}^{x,y}(a,b)=H_{0,2}^{x,y}(a,b)=H_{1,2}^{x,y}(a,b)=0$ forces $Q(a,b,z) \equiv 0 \quad (mod  L(a,b,z))$ and hence $ X_{a,b}^z$ is degenerate containing the entire line $L_{a,b}^z$.\\
If $L(a,b,z) \neq 0$ and the fibre $X_{a,b}^z$ is degenerate we must have $L_{a,b}^z \subset Q_{a,b}^z$.  We are going to proof that $G_0^{x,y}(a,b)=G_1^{x,y}(a,b)=G_2^{x,y}(a,b)=H_{0,1}^{x,y}(a,b)=H_{0,2}^{x,y}(a,b)=H_{1,2}^{x,y}(a,b)=0$.  First let's do $G_0^{x,y}(a,b)=0$.  If $L^{x,y}_1(a,b)=L^{x,y}_2(a,b)=0$, this follows from the definition, otherwise $(0,L^{x,y}_2(a,b),-L^{x,y}_1(a,b)) \in L^z_{a,b}$ and therefore must belong to $Q_{a,b}^z $, when we evaluate we get
$$0=Q_{1,1}^{x,y}(a,b) (L^{x,y}_2(a,b))^2 - Q_{1,2}^{x,y}  L^{x,y}_2(a,b)L^{x,y}_1(a,b)+ Q_{2,2}^{x,y}(L^{x,y}_1(a,b))^2$$
So $G^{x,y}_0(a,b)=0$.  In a similar way we do $G_1^{x,y}(a,b)=G_2^{x,y}(a,b)=0$.  The substitution of the results $G_i^{x,y}(a,b)=0$ in the equations and evaluations at $x=a, y=b$ will give
$$H_{1,2}^{x,y}(a,b) z_1 z_2=H_{0,2}^{x,y}(a,b) z_0 z_2=H_{1,0}^{x,y}(a,b) z_1 z_0=0$$
for all points $(z_0,z_1,z_2) \in L^z(a,b)$.  If $L^z(a,b)$ is the line $z_1=0$, then $L_0^{x,y}(a,b)=L_2^{x,y}(a,b)=0$ and $H_{1,2}^{x,y}(a,b)=0$ using the definition. If $L^z(a,b)$ is the line $z_2=0$, then $L_1^{x,y}(a,b)=L_2^{x,y}(a,b)=0$ and $H_{1,2}^{x,y}(a,b)=0$ will be again equal to zero.  Otherwise if $L^z_{a,b}$ is none of the lines $z_1=0$ or $z_2=0$, then $H_{1,2}^{x,y}(a,b)=0$ from the previous line.  The other cases for $H^{x,y}_{i,j}(a,b)=0$ are solved similarly.
\end{proof}
We can now define open sets $U_1,U_2,U_3$ in such a way that the dominant rational maps $\sigma_i : X \dashrightarrow X$ are bijective morphisms $$\sigma_i : U_i \lra U_i.$$
\begin{align*}
U_1=X-\{ (a,b,c) \in X & : G_0^{y,z}(b,c)=   G_1^{y,z}(b,c)=G_2^{y,z}(b,c)=0 \\ & H_{0,1}^{y,z}(b,c)=H_{0,2}^{y,z}(b,c)=H_{1,2}^{y,z}(b,c)=0\},\\
U_2=X-\{ (a,b,c) \in X & : G_0^{x,z}(a,c)=  G_1^{x,z}(a,c)=G_2^{x,z}(a,c)=0 \\ & H_{0,1}^{x,z}(a,c)=H_{0,2}^{x,z}(a,c)=H_{1,2}^{x,z}(a,c)=0\},\\
U_3=X-\{ (a,b,c) \in X & : G_0^{x,y}(a,b)= G_1^{x,y}(a,b)=G_2^{x,y}(a,b)=0 \\ & H_{0,1}^{x,y}(a,b)=H_{0,2}^{x,y}(a,b)=H_{1,2}^{x,y}(a,b)=0  \}.
\end{align*}
The maps $\sigma_1, \sigma_2, \sigma_3$ induce maps on divisors: Let's consider $Y$ a closed subvariety of codimension one and $\sigma_i^*Y= \overline{\sigma_i^{-1}Y}$, the Zariski closure of the pre-image.  In this way we induce maps on Weil divisors, that respect linear and numerical equivalence and descend to maps $$\sigma_i^* : \Pic(X) \lra \Pic(X) \qquad \tilde{\sigma}_i^* : NS(X)_\Q \lra NS(X)_\Q.$$ To study the action of the $\sigma_i^*$ on $\Pic(X)$ we denote by $H,H'$ hyperplane sections representing the two fundamental classes in $\Pic(\Pl^2 \times \Pl^2)$,
\begin{align*}
H & =\{((a_0:a_1:a_2),(b_0:b_1:b_2)) \in \Pl^2 \times \Pl^2 : a_0=0 \},\\
H' &=\{((a_0:a_1:a_2),(b_0:b_1:b_2)) \in \Pl^2 \times \Pl^2 : b_0=0 \}.
\end{align*}
and the divisors $D_x,D_y,D_z$ on $X$ defined by:
$$ D_x = \{ P \in X : x_0=0\},\quad
D_y = \{ P \in X : y_0=0\}, \quad
D_z= \{ P \in X : z_0=0\}. $$
The pullbacks of $H,H'$ by the different projections give back the $D_x,D_y,D_z$,
$$
p_{xy}^*H = p_3^* H=   D_x, \quad
p_{xy}^*H'=  p_3^* H'=    D_y, $$
$$ p_{xz}^*H=  p_2^* H=    D_x, \qquad
p_{xz}^*H'=  p_2^* H'=   D_z, $$
$$p_{yz}^*H=  p_1^* H=   D_y, \qquad
p_{yz}^*H'=  p_1^* H'=   D_z. $$

\begin{lemma} \label{projections} We have the following equivalences of divisors in $\Div(X)$:
\begin{enumerate}
\item $p_{1*} p_2^* H \sim 4H+4H'; $
\item $p_{2*} p_1^* H \sim 4H+4H'; $
\item $p_{3*} p_1^* H' \sim 4H+4H';$
\end{enumerate}
\end{lemma}
\begin{proof}  The prove of all parts will be analogous and straightforward from the definition of $H,H'$ and the $p_i$'s.  Let's see for example the proof of (i).  The pull-back $p_2^* H = \{ P \in X : x_0=0\}$ is given by the two equations
\begin{equation*} \label{proj}
L_1^{y,z} x_1 +  L_2^{y,z} x_2=0, \qquad Q_{1,1}^{y,z} x_1^2 + Q_{1,2}^{y,z} x_1 x_2 + Q_{2,2}^{y,z}x_2^2=0.
\end{equation*}
When we project onto $(y,z)$ we eliminate $x_1, x_2$ and get the equation $$G_0^{y,z}=(L_1^{y,z})^2Q^{y,z}_{2,2}-L_1^{y,z}L_2^{y,z}Q^{y,z}_{1,2}+(L_2^{y,z})^2Q^{y,z}_{1,1}=0.$$
where $G_0^{y,z}$ is a $(4,4)$-bihomogeneous form in $y$ and $z$, and therefore $p_{1*} p_2^* H \sim 4H+4H'$. \end{proof}

Applying lemma \ref{projections} we obtain the pushforwards:
$$p_{1*}(D_x) \sim  4H+4H', \qquad  p_{2*}(D_y) \sim 4H+4H', \qquad p_{3*}(D_z) \sim 4H+4H',$$

and the action of the $\sigma^*_i$'s on the divisors $D_x,D_y,D_z$:
$$\sigma_1^*(D_x)=p_1^*p_{1*} D_x -D_x \sim 4D_y+4D_z-D_x ,$$
$$\sigma_1^*(D_y)=\sigma_1^* p_1^* H = (p_1 \circ \sigma_1)^* H =  D_y,$$
$$\sigma_1^*(D_z)=\sigma_1^* p_1^* H'=(p_1 \circ \sigma_1)^* H' = D_z,$$
$$ \sigma_2^*(D_x)=\sigma_2 p_2^*H= (p_2 \circ \sigma_2)^* H' = D_x,$$
$$\sigma_2^*(D_y)=p_2^* p_{2*} D_y - D_y \sim 4D_x + 4D_z - D_y,$$
$$\sigma_2^*(D_z)=\sigma_2^*p^*_2H'=(p_2 \circ \sigma_2)^* H' = D_z,$$
$$\sigma_3^*(D_x)=\sigma_3^* p_3^* H =(p_3 \circ \sigma_3)^* H = D_x,$$
$$\sigma_3^*(D_y)=\sigma_3^* p_3^* H' =(p_3 \circ \sigma_3)^* H' = D_y,$$
$$\sigma_3^*(D_z)=p^*_3 p_{3*} D_z - D_z \sim 4D_x+4D_y-D_z.$$
Using the actions of the $\sigma^*_i$ we can get a polarizations by a very ample line bundle for the system of involutions $\sigma_1,\sigma_2,\sigma_3$.
\begin{proposition}
Suppose that $r_x,r_y,r_z $ are positive real numbers and we have the polarization by three maps
$$\sum_i \sigma_i^*(r_x D_x + r_yD_y + r_z D_z) \sim d (r_x D_x + r_yD_y + r_z D_z),$$
in $\Pic(X) \otimes \R$.  Then $d=9$ and $r_x=r_y=r_z=1$.
\end{proposition}
\begin{proof}
When we add up the actions of $\sigma_i^*$ on $r_x D_x + r_yD_y + r_z D_z$, and equal that to $d(r_x D_x + r_yD_y + r_z D_z)$ for some $d > 3$, we get the system of linear equations:
\begin{align*}
r_x+4r_y+4r_z & = dr_x, \\
4r_x+r_y+4r_z & = dr_y, \\
4r_x+4r_y+r_z & = dr_z.
\end{align*}
The determinant is $(9-d)(3+d)^3$ and the value of $d=9$ gives $r_x=r_y=r_z=1$.
\end{proof}

\begin{proposition}
The maps $\sigma_i$ and $\sigma_{ij}=\sigma_i \circ \sigma_j$, for $i,j \in \{0,1,2\}$, satisfy the properties:
\begin{itemize}
\item[(1)] $(\sigma_i \circ \sigma_j)^*=\sigma^*_j \circ \sigma^*_i$,
\item[(2)] $ (\sigma_{ij}^{n})^*=(\sigma_{ij}^{*})^n$.
\end{itemize}
\end{proposition}
\begin{proof}
In general, given two rational maps $\tau : X \dashrightarrow X$ and $\tau' : X \dashrightarrow X$ defining involutions $\tau : U_{\tau} \lra U_{\tau}$ and $\tau' : U_{\tau'} \lra U_{\tau'}$ on open sets $U_{\tau}$ and $U_{\tau'}$ respectively, we will have $(\tau \circ \tau')^*=\tau'^* \circ \tau^*$.  Let $Y$ be an irreducible subvariety.  If $P \in \overline{\tau(Y \cap U_{\tau})} \cap U_{\tau'}$, there exist a sequence $P_n \rightarrow P$, with $P_n \in \tau(Y \cap U_{\tau}) \cap U_{\tau'}$.  Therefore $\tau'(P_n) \rightarrow \tau'(P)$ and $\tau'(P) \in \overline{\tau' (\tau (Y \cap U_{\tau})) \cap U_{\tau'})}$.  In other words $\overline{\tau'(\overline{\tau (Y \cap U_{\tau})} \cap U_{\tau'})} \subset \overline{\tau'(\tau (Y \cap U_{\tau})\cap U_{\tau'})}$, so this two sets must be equal and $(\tau \circ \tau')^*=\tau'^* \circ \tau^*$.  For the first part of the theorem we take $\sigma_i=\tau$ and $\sigma_j=\tau'$.  For the second part we proceed by induction and use the result to proof the induction step.  If we suppose that $ (\sigma_{ij}^{n})^*=(\sigma_{ij}^{*})^n$ is true, then $ (\sigma_{ij}^{*})^{n+1}=\sigma_{ij}^*((\sigma_{ij}^*)^n)=\sigma_{ij}^*((\sigma_{ij}^n)^*)$
By our result above the last equals to $(\sigma_{ij}^{n+1})^{*}$. \end{proof}

\subsection{Computation of dynamical degree}
In this subsection we study the action induced by the maps $\sigma_{ij}=\sigma_i \circ \sigma_j$ on the subspace  $V=\Span(D_x,D_y,D_z)$ of $\Pic(X) \otimes \R$.  As an application we will be able to get the dynamical degree of those maps for members of the family with Picard number $p(X)=3$.
\begin{theorem}
Let $\sigma_{ij}$ be the rational dominant map $\sigma_i \circ \sigma_j : X \dashrightarrow X$. Let $V$ be the subspace of $\Pic(X) \otimes \R$ spanned by $D_x, D_y, D_z$ and consider the action of $\sigma_{ij}^{*n} : V \lra V$.  The eigenvalues of $\sigma_{ij}^{*n}|V$ belong to the set $\{1,\beta^n, \beta'^{n} \}$, where $\beta=7+4\sqrt{3}$ and $\beta'=\frac{1}{\beta}$.
\end{theorem}
\begin{proof}
The action of the maps $\sigma_{12}^*$, $\sigma_{31}^*$, $\sigma_{31}^*$, $\sigma_{32}^*$, $\sigma_{13}^*$ and $\sigma_{23}^*$ with respect to that base $\{D_x,D_y,D_z\}$ is given respectively by the matrices
$$\sigma^*_{12}=\left(\begin{array}{rrr}
-1 & -4 & 0 \\
4 & 15 & 0 \\
4 & 20 & 1
\end{array}\right) \qquad \sigma^*_{13}=\left(\begin{array}{rrr}
15 & 0 & 4 \\
20 & 1 & 4 \\
-4 & 0 & -1
\end{array}\right)$$
$$ \sigma^*_{12}=\left(\begin{array}{rrr}
15 & 4 & 0 \\
-4 & -1 & 0 \\
20 & 4 & 1
\end{array}\right) \qquad \sigma^*_{23}=\left(\begin{array}{rrr}
1 & 20 & 4 \\
0 & 15 & 4 \\
0 & -4 & -1
\end{array}\right)$$
$$\sigma^*_{31}=\left(\begin{array}{rrr}
-1 & 0 & -4 \\
4 & 1 & 20 \\
4 & 0 & 15
\end{array}\right) \qquad \sigma^*_{32}=\left(\begin{array}{rrr}
1 & 4 & 20 \\
0 & -1 & -4 \\
0 & 4 & 15
\end{array}\right)$$
With the help of SAGE we find that the six matrices are sharing the same characteristic polynomial $p=-(\lambda-1)(\lambda^2-14\lambda+1)$.  The roots of $p(\lambda)$ are $\{1,\beta, \beta' \}$ with $\beta=7+4\sqrt{3}$ and $\beta'=1/\beta$, therefore all the six matrices are diagonalizable and the eigenvalues of the the powers are from the set $\{1,\beta^n,\beta'^n\}$.
\end{proof}

\begin{corollary}
Suppose that the Picard number $p(X)=3$, then the first dynamical degree of $\sigma_{ij}$,
$\delta_{\sigma_{ij}}=\beta.$

\end{corollary}
\begin{proof}
The divisors $D_x,D_y,D_z$ represent three distinct classes in $NS(X)_\Q$.  If the Picard number $p(X)=3$, then we have $NS(X)_\Q \cong V_\Q$.  The first dynamical degree of any of the maps $\sigma_{ij}$ is:
$$ \delta_{\sigma_{ij}}= \limsup_{n \rightarrow \infty} \rho((\sigma_{ij}^n)^*)^{1/n}=\limsup_{n \rightarrow \infty} \rho((\sigma_{ij}^*)^n)^{1/n}=\limsup_{n \rightarrow \infty} (\beta^n)^{1/n}=\beta .$$ \end{proof}

\end{document}